\documentclass[11pt]{article}
\usepackage{pifont}

\usepackage{amsfonts}
\usepackage{latexsym}
\usepackage{amsmath}
\usepackage{amssymb}
\usepackage{color}

\def\dsp{\displaystyle}
\def\3n{\negthinspace \negthinspace \negthinspace }

\newtheorem{theorem}{Theorem}[section]

\newenvironment{proof}[1][Proof]{\noindent \textbf{#1.} }{\ \ \  $\Box$}

\newtheorem{lemma}{Lemma}[section]
\newtheorem{definition}{Definition}[section]
\newtheorem{remark}{Remark}[section]

\def\dbE{\mathbb{E}}

\def\dbR{\mathbb{R}}

\def\={\buildrel \triangle \over =}

\title{Mean Field Linear-Quadratic-Gaussian (LQG) Games for Stochastic Integral Systems}

\date{}

\begin{document}

\author{Jianhui Huang\thanks{J. Huang is with the Department of Applied Mathematics,
The Hong Kong Polytechnic University, Hong Kong (E-mail:
majhuang@polyu.edu.hk).} \quad \quad Xun Li \thanks{X. Li is with
with the Department of Applied Mathematics, The Hong Kong
Polytechnic University, Hong Kong (E-mail: malixun@polyu.edu.hk).}
\quad \quad Tianxiao Wang \thanks{Corresponding author, T. Wang is
with the School of Mathematics, Sichuan University, Chengdu, P. R.
China (E-mail:xiaotian2008001@gmail.com).}}

\maketitle

\begin{abstract}
In this paper we discuss a class of mean field
linear-quadratic-Gaussian (LQG) games for large population system
which has never been addressed by existing literature. The features
of our works are sketched as follows. First of all, our state is
modeled by stochastic Volterra-type equation which leads to some new
study on stochastic ``integral" system. This feature makes our
setup significantly different from the previous mean field games where the states
always follow some stochastic ``differential" equations.
Actually, our stochastic integral system is rather general and can
be viewed as natural generalization of stochastic differential
equations. In addition, it also includes some types of stochastic
delayed systems as its special cases. Second, some new techniques
are explored to tackle our mean-field LQG games due to the special
structure of integral system. For example, unlike the Riccati
equation in linear controlled differential system, some
Fredholm-type equations are introduced to characterize the
consistency condition of our integral system via the resolvent
kernels. Third, based on the state aggregation technique, the Nash
certainty equivalence (NCE)
equation is derived and the set of associated decentralized controls are verified to satisfy the $\epsilon$-Nash equilibrium property. To this end, some new estimates of stochastic Volterra equations are developed which also have their own interests.\\

\end{abstract}



\section{Introduction}\label{sec:intro}

The large-population systems have
important applications in various fields including, but not limited
to, social science \cite{H-C-M 2009}, financial management, economics \cite{L 1984}, \cite{W-B-V 2005}, industry engineering (e.g., multi-agent systems
\cite{L-Z 2008}, \cite{W-Z 2012}, coupled oscillators
\cite{Y-M-M-S 2012}, wireless communication \cite{H-C-M 2004}), etc.
The most significant feature of large population system lies in the existence of considerable negligible agents or players involved in their individual states or (and) cost functionals through the state average across the whole population. Due to such high complexity in dynamics modeling, it is intractable to study the ``centralized" control strategy since it strongly requires the congregation of global information throughout all individual agents. In many cases, such information congregation procedure turns out to be impossible yet unnecessary due to the complex interaction of all agents. Alternatively, it will be more reasonable to study the ``decentralized" decision strategies of our controlled large population systems for sake of complexity deduction and computation efficiency.

Recently, there has been increasing research interest in studying
this type of stochastic control problems as well as addressing their
applications. In particular, the research in this area has been well
developed for the decentralized control of large population system
during the last few decades. Among them, one efficient methodology
is to consider the related mean-field game by transforming the
large-population or multi-agent optimization into one single-agent
optimization. Concerning to this methodology, for each given agent,
only decentralized information of its own state is required to
design the control policy while the effects of remaining agents are
summarized by some mean-field term to be determined by the
consistency condition. As a consequence, each agent aims to solve
some Hamiltonian-Jacobi-Bellman (HJB) equation which is coupled with
the Fokker-Planck (FP) equation of controlled large population
system. By the coupling of FP equation, the population or global
effects are combined thus each agent only needs to derive the
feedback strategy involving its own state. In case of the linear
controlled system, it leads to some LQG mean-field games where the
Riccati equations are coupled with the FP equation of our system.
Some relevant literature of LQG mean-field game include \cite{B
2013}, \cite{H 2010}, \cite{H-M-C 2006}, \cite{H-C-M 2007},
\cite{L-L 2007}, \cite{N-H 2012} and the reference therein. In rough
sense, its general setup can be presented as follows. There exist a
collection of large number negligible players where individual state
follows linear stochastic differential equations
\begin{equation}\label{eq:1-1}
 dx_i(t)=\big[A_i x_i(t)+B_i u_i(t)+F_i x^{(N)}(t)\big]dt+D_idW_{i}(t),
 \end{equation}
where $x^{(N)}=\frac{1}{N}\sum\limits_{i=1}^{N}x_i$ denotes the
state average or mean
 field term. It characterizes the global effects of all agents in our population. The cost
 functional to be minimized takes the following form:
\begin{eqnarray*}
 J_{i}(u_i(\cdot)) = \frac{1}{2}\mathbb{E}\left\{\int_0^{T}\Big[Q(x_i(t)-x^{(N)}(t))^2
+Ru^2_i(t)\Big]dt+Hx_i^{2}(T)\right\}.
\end{eqnarray*}
The scheme of mean-field LQG game is formulated into some limiting LQG problem
 which can be solved using the own individual state without the complicated interactions
 of all agents. Based on it, by reproducing the best response of the state,
 the decentralized control policy can be derived and some Nash equilibrium property
  in near-optimal sense should be verified. Last, we would
  like to mention some other works on mean-field (McKean-Vlasov) SDEs (MFSDEs): for example, see \cite{AD 2010} \cite{BDL 2011} for stochastic mean-field control and related maximum principles, \cite{BDF 2012} for the large deviation principle (LDP) of weakly coupled MFSDEs,
  and \cite{MOZ 2012} for the stochastic control of MFSDEs using the Malliavian calculus and its applications to mean-variance problem.

Based on above works, herein we turn to study
the large population system with stochastic ``integral" instead of
``differential" dynamics. Accordingly, new methodologies are
required to be developed, which are significantly different from
those in the standard stochastic control theory in the literature,
to cope with the mathematical difficulties involved in this
mean-field LQG. Our motivation to mean-field LQG integral games are
mainly as follows.
 First, it is remarkable there exists many stochastic or deterministic models arising from physical,
economics and finance which cannot be represented by differential
systems, see for example, stochastic input-output model \cite{D-H
1986}, capital replacement model \cite{K-M 1976}, models in
nanoscale biophysics \cite{K 2008}. Therefore,
 in terms of stochastic integral system, it is more appropriate to study the
  mean-field LQG integral game emerged naturally in the mass behavior of
  large population. Second, our integral system includes some cases of stochastic delay
in state or control variable which play important roles in biology,
social science and engineering. One example is advertising model for
which there are some delay effects (\cite{G-M 2006}, \cite{G-M-S
2009}). When we consider the interactions effects of all small
advertising firms through the whole industry, it is necessary to
formulate some mean-field integral games. Another example is the
dynamic optimization of large population wireless interaction knots
where there exist some delay effects in signal transmission thus
some large population stochastic delayed games should be formulated.

Motivated by above concerns, in this paper, we consider the following controlled stochastic
Volterra integral equations
\begin{eqnarray}\label{eq:1-3}
x_{i}(t)&=&\varphi(t)+\int_{0}^{t}b(t,s)x_{i}(s)ds+\int_{0}^{t}f(t,s)x^{N}(s)ds\nonumber \\
&&+\int_{0}^{t}c(t,s)u_{i}(s)ds+%
\int_{0}^{t}\sigma(t,s)dW_{i}(s).
\end{eqnarray}
where $x^{N}=\frac{1}{N}\sum^{N}_{j=1}x_j$ is the mean field term
which characterizes average interaction and mass effects of our
population in \emph{spatial} variable, $\varphi,b,c,\sigma$ are
deterministic functions, called the kernels, which characterize path
dependence of out state in its \emph{temporal} variable. Compared
with mean field LQG games in \cite{B 2013}, \cite{H
2010}-\cite{H-C-M 2009}, \cite{N-H 2012}, there are several new
features or obstacles arising in such general framework. The first
one is concerned on the expression of optimal control where the
skills of Riccati equations are no long useful. By means of series
skills, \emph{resolvent kernel}, and backward stochastic Volterra
integral equations (BSVIEs for short), we obtain optimal control
$\widehat{u}_i(t)$ where the path dependence on state $x_i$ in
$[t,T]$ is shown explicitly. As to BSVIEs, we refer the readers to
\cite{S-W-Y 2013}, \cite{W-S 2010}, \cite{Y 2006}, \cite{Y 2008}.
Secondly, the equation satisfied by optimal state is no long a
linear SDEs but a new kind of equation with double integrals. We
call it stochastic Volterra-Fredholm equation, the solvability of
which is discussed under certain conditions. Thirdly, in the
procedure of deriving NCE equation, certain Fredholm integral
equations work well while Riccati equation becomes unapplicable
again. Such result is consistent with the arguments in \cite{Y 1996}
where Riccati equation is replaced by Fredholm integral equation to
express result more clearly. Fourthly, in order to discuss
asymptotic equilibrium analysis of decentralized control, some
nontrivial extension of classical results are needed, and new tricks
in doing these are developed. In addition, the optimal control,
optimal state and NCE equation are derived in a new manner.

The rest of this paper is organized as follows. In Section 2, the
mean field LQG game for stochastic integral system is formulated. Section 3
is devoted to the discussion of the NCE equation and the consistency condition. The $\epsilon$-Nash equilibrium property of decentralized strategy is also discussed therein. In Section 4, some special cases are discussed. Section 5 concludes our work.

\section{Problem formulation}

In this paper, we set the state equation to be one-dimensional for sake of notation simplicity. There has no essential difficulty to extend the results to the multi-dimensional case. Suppose $(\Omega, \mathcal{F}, (\mathcal{F}_t)_{0\leq t\leq T},
\mathbb{P})$ is a complete filtered probability space on which
$W(t)=(W_i(t))_{0 \leq t \leq T}$, $1\leq i\leq N$ are independent scalar-valued Brownian motions. It is also assumed $(\mathcal
F_t)_{0\leq t\leq T}$ is the natural filtration generated by $W(t)$
and $\mathcal F=\mathcal F_T$. We also need to introduce the
following spaces.
\begin{eqnarray*}
\begin{split}
L^2(0,T;\dbR)=&\Big\{x:[0,T]\to
\dbR\Bigm|x(\cdot)\ \text{\rm is deterministic such that}\int_0^T|x(s)|^2ds<\infty\Big\},\\
L_\mathcal{F}^2(0,T;\dbR)=&\Big\{x:[0,T]\times\Omega\to \dbR\Bigm|
x(\cdot)\hbox{ is
$\mathcal{F}_t$-adapted process such that }
\dbE\int_0^T|x(s)|^2ds<\infty\Big\},\\
C(0,T;\dbR)=&\Big\{x:[0,T]\to \dbR\Bigm| x(\cdot)\hbox{ is
deterministic and continuous on $[0,T]$} \Big\}.
\end{split}
\end{eqnarray*}
Now consider the stochastic large population integral system which consists of $N$ individual negligible players, denote respectively
by $\mathcal{A}_i, i=1, 2, \cdots, N$. The dynamics of
$\mathcal{A}_i$ is given by the following stochastic Volterra
equation
\begin{eqnarray}\label{2.1}
x_{i}(t) & = & \varphi(t)+\int_{0}^{t}b(t,s)x_{i}(s)ds+\int_{0}^{t}f(t,s)x^{N}(s)ds\nonumber\\
& & +\int_{0}^{t}c(t,s)u_{i}(s)ds+%
\int_{0}^{t}\sigma(t,s)dW_{i}(s).
\end{eqnarray}%
 All
individual players are coupled in terms of their individual cost functional
as follows:
\begin{equation}\label{2.2}
\mathcal{J}_{i}(u_{i},u_{-i})=\dbE\int_{0}^{T}\left[
(x_{i}(t)-\gamma x^{N}(t)-\eta )^{2}+Ru_{i}^{2}(t)\right] dt,
\end{equation}%
where $u_{-i}=(u_1,\cdots,u_{i-1},u_{i+1},\cdots,u_{N})$, and $R>0$,
$\gamma$, $\eta$ are constants. Some special cases of our general state equations are as
follows.

The first one is the differential case, i.e., the kernels $\varphi,b,c,\sigma $
are independent of $t.$ In this case, equation (\ref{2.1}) becomes the following
stochastic \emph{integral} equation,
\begin{eqnarray}\label{2.3}
x_{i}(t)=x(0)+\int_{0}^{t}b(s)x_{i}(s)ds+%
\int_{0}^{t}c(s)u_{i}(s)ds+\int_{0}^{t}\sigma(s)dW_i(s).
\end{eqnarray}
Note that all the previous mentioned literature of large population
LQG games are based on equation (\ref{2.3}). The second case is the
stochastic differential equation with delay in state which can be
described as
\begin{equation}\label{2.4}
dx_{i}(t)=\left[
A(t)x_{i}(t-h)+\int_{t-h}^{t}B(t,s)x_{i}(s)ds\right]
dt+C(t)u_i(t)dt+D(t)dW_i(t),
\end{equation}%
where $h>0$ is the delay or lag parameter, $x_{i}(t)=k(t),$ $t\in \lbrack -h,0].$ The third one
includes stochastic differential equations with delay in control
variable which evolves as
\begin{equation}\label{2.5}
dx_{i}(t)=A(t)x_{i}(t)dt+C(t)u_i(t-h)dt+D(t)dW_i(t),
\end{equation}%
where $h>0,$ $x_{i}(t)=k(t),$ $t\in \lbrack -h,0].$ Under certain
conditions, we can transform equation (\ref{2.4}) and (\ref{2.5})
into the following form of stochastic Volterra integral equation,
\[
x_{i}(t)=\psi(t)+\int_{0}^{t}K(t,s)u_i(s)ds+\int_{0}^{t}\Phi
(t,s)D(s)dW_i(s).
\]%
We refer to Section 4 for more details on these procedures.

\section{NCE equation system and asymptotic equilibrium analysis}

Given the state dynamics (\ref{2.1}) and cost functional
(\ref{2.2}), this section aims to derive the
associated Nash certainty equivalence (NCE) equation. To this end, some decentralized asymptotic equilibrium analysis should be given.

\subsection{NCE equation}

First, let us formulate one auxiliary limiting control problem via
the approximation of state average $x^{N}$ by some deterministic
function $a(t)$. Now we introduce the following state
equation
\begin{eqnarray}\label{3.1}
x_{i}(t)&=&\varphi
(t)+\int_{0}^{t}b(t,s)x_{i}(s)ds+\int_{0}^{t}f(t,s)a(s)ds\nonumber
\\
&&+\int_{0}^{t}c(t,s)u_{i}(s)ds+%
\int_{0}^{t}\sigma(t,s)dW_{i}(s),\quad t\in[0,T],
\end{eqnarray}
and the cost functional for player $\mathcal{A}_i$
($i=1,2,\cdots,N$),
\begin{eqnarray}\label{3.2}
\overline{J}_{i}(u_{i})=\dbE\int_{0}^{T}\Big[(x_{i}(t)-\gamma a(t)-\eta )^{2}+R|u_{i}(t)|^{2}\Big] dt.
\end{eqnarray}
In this case, it follows from \cite{S-W-Y 2013} or \cite{Y 2008} that the optimal control $\overline{u}_{i}$ can be
represented by
\begin{equation}\label{3.3}
\overline{u}_{i}(t)=-\frac{1}{2R}\dbE^{\mathcal{F}_{t}}\!\!\int_{t}^{T}c(s,t)
\overline{y}_{i}(s)ds
\end{equation}%
where $\overline{y}_i$ satisfies the following linear Backward Stochastic Volterra Integral Equation (BSVIE)
\begin{equation}\label{3.4}
\overline{y}_{i}(t)=2\overline{x}_{i}(t)-2\gamma a(t)-2\eta
+\int_{t}^{T}b(s,t)\overline{y}_{i}(s)ds-\int_{t}^{T}\overline{z}%
_{i}(t,s)dW_{i}(s),
\end{equation}%
and $\overline{x}_{i}$ is the solution of (\ref{3.1}) with respect
to $\overline{u}_i$.
Due to the special form of (\ref{3.4}), by using Lemma 1.1 in \cite{Y 1996} we can formulate $\overline{y}_{i}$ as%
\begin{equation}\label{3.5}
\overline{y}_{i}(t)=2\overline{x}_{i}(t)-2\gamma a(t)-2\eta
+\int_{t}^{T}2P(s,t)[\overline{x}_{i}(s)-\gamma a(s)-\eta ]ds,
\end{equation}%
where for any $0\leq t\leq s\leq T,$
\begin{eqnarray}\label{3.5.1}
P(s,t) =\sum\limits_{k=1}^{\infty }\Lambda^{k}(s,t), \quad
\Lambda^{1}(s,t) =b(s,t), \quad
\Lambda^{k+1}(s,t)=\int_{t}^{s}b(s,r)\Lambda^{k}(r,t)dr.
\end{eqnarray}%
\begin{remark}\label{remark3.1}
Note that (\ref{3.5}) can be viewed as a constant variation formula
in our setup and $P$ is called ``resolvent kernel". In
the special multi-dimensional differential case, $P$ becomes the
solution of fundamental matrix. The method of infinite series will be
frequently adopted in our following analysis. On the other hand, it is straightforward to
verify%
\[
|\Lambda^{k}(s,t)|<\frac{M^{k}|t-s|^{k-1}}{(k-1)!},\quad
t,s\in[0,T],\quad k\geq 1,
\]%
where $M$ is the upper bound of $b$, thus $P$ is bounded. Moreover,
if $b$ is continuous in $t,$ then one can see that $P$ is also
continuous in $t.$
\end{remark}
Substituting (\ref{3.5}) into (\ref{3.3}), we can rewrite the
optimal control in (\ref{3.3}) as
\begin{equation}\label{3.6}
\overline{u}_{i}(t)=-\frac{1}{R}\dbE^{\mathcal{F}_{t}}\!\!\int_{t}^{T}\widehat{c}(s,t)[\overline{x}_{i}(s)-\gamma a(s)-\eta ]ds,
\end{equation}%
where
\begin{eqnarray}\label{3.7}
\widehat{c}(s,t)=c(s,t)+\int_{t}^{s}P(s,r)c(r,t)dr,\quad s\geq t.
\end{eqnarray}
Note that if $b$, $c$ are continuous in $t,$ by using Remark \ref{remark3.1} one can see that $\widehat{c}%
(t,s)$ is also continuous in $t.$ On the other hand, by using Lemma
1.1. of \cite{Y 1996}, we can rewrite (\ref{3.1}) as%
\begin{equation}\label{3.8}
x_{i}(t)=\widehat{\varphi }(t)+\widehat{\sigma }_i(t)+\int_0^t\widehat{f}(t,s)a(s)ds+\int_{0}^{t}\widehat{c}(t,s)u_{i}(s)ds,
\end{equation}%
where $P$ is defined in (\ref{3.5.1}) and
\begin{align}\label{3.8.1}
\widehat{\varphi }(t) &=\varphi (t)+\int_{0}^{t}P(t,s)\varphi(s)ds, \nonumber\\
\widehat{\sigma }_{i}(t) &=\int_{0}^{t}\left[\sigma(t,s)+\int_{s}^{t}P(t,r)\sigma(r,s)dr\right] dW_{i}(s),\nonumber\\
\widehat{f}(t,s)&=f(t,s)+\int_{s}^{t}P(t,r)f(r,s)dr.
\end{align}%
Consequently, by combining (\ref{3.6}) and (\ref{3.8}), as well as
stochastic Fubini theorem, the optimal state equation of this limit
control problem can be described by
\begin{align}\label{3.9}
\overline{x}_{i}(t) = & \; \widehat{\varphi }(t)+\widehat{\sigma}_i
(t)-\frac{1}{R}\int_{0}^{T}\int_{0}^{s\wedge t}\widehat{c}(t,r)\widehat{c}%
(s,r)\dbE^{\mathcal{F}_{r}}\overline{x}_{i}(s)drds  \nonumber\\
& + \int_0^t\widehat{f}(t,s)a(s)ds+\frac{1}{R}\int_{0}^{T}\int_{0}^{s\wedge t}
\widehat{c}(t,r)\widehat{c}(s,r)dr\cdot[\gamma a(s)+\eta ]ds.
\end{align}%
In particular, by taking expectation,
\begin{equation}\label{3.10}
\dbE\overline{x}_{i}(t)=\widehat{\varphi }(t)+\int_0^t\widehat{f}(t,s)a(s)ds+\frac{1}{R}
\int_{0}^{T}M(t,s)[\gamma a(s)+\eta-\dbE\overline{x}_{i}(s)]ds,
\end{equation}%
where
\begin{eqnarray}\label{3.11}
M(t,s)=\int_{0}^{s\wedge t}\widehat{c}(t,r)\widehat{c}(s,r)dr,\quad
t,s\in[0,T].
\end{eqnarray}
Given $a(s)$, (\ref{3.10}) can be regarded as a Fredholm integral
equation with solution being $\dbE x_i(\cdot)$. Before analyzing,
we present one result on the solvability of Fredholm
integral equation which is easy to check.
\begin{lemma}\label{lemma-1}
Consider the following Fredholm integral equation
\begin{equation}\label{3.12}
x(t)=\psi (t)+\frac{1}{R}\int_{0}^{T}A(t,s)x(s)ds,\quad t\in[0,T],
\end{equation}%
where $\psi $, $A$ are bounded and continuous functions in $t$. If
\[
\sup_{t\in [0,T]}\int_{0}^{T}|A(t,s)|^{2}ds<R, \quad
\sup_{t\in [0,T]}|\psi (t)|^{2}<\infty ,
\]%
then there exists a unique continuous solution $x(\cdot)$ of (\ref{3.12}).
\end{lemma}
From Lemma \ref{lemma-1}, we can get the following result directly.
\begin{theorem}\label{theorem-1}
Suppose $\varphi,$ $b$, $f$, $c$ are bounded and continuous in $t$
such that
\[
\sup_{t\in \lbrack 0,T]}\int_{0}^{T}|M(t,s)|^{2}ds<R.
\]
Then, given $a\in C[0,T]$, there exists a
unique continuous function $\dbE \overline{x}_i(\cdot)$ satisfying
equation (\ref{3.10}).
\end{theorem}
Given $a(s)$, using Theorem \ref{theorem-1}, we can define $\Gamma a$
by
\begin{equation}\label{3.13}
\lbrack \Gamma a](t)=\widehat{\varphi
}(t)+\int_0^t\widehat{f}(t,s)a(s)ds
+\frac{1}{R}%
\int_{0}^{T}M(t,s)\big\{\gamma a(s)+\eta-[\Gamma a](s)\big\}ds.
\end{equation}%
The Nash certainty equivalence (NCE) equation in our setting is thus
\begin{equation}\label{3.14}
\widehat{a}(t)=\widehat{\varphi}(t)+\int_0^t\widehat{f}(t,s)
\widehat{a}(s)ds+\frac{1}{R}\int_{0}^{T}M(t,s)[(\gamma-1)\widehat{a}(s)+\eta ]ds,
\end{equation}%
where $\widehat{a}$ is the corresponding solution. Note that
(\ref{3.14}) is a linear deterministic Volterra-Fredholm integral
equation. Before going further, similar to (\ref{3.8}) above, we
transform (\ref{3.14}) into a Fredholm integral equation. More
precisely, by using Lemma 1.1. in \cite{Y 1996}, we can rewrite
(\ref{3.14}) as
\begin{equation}\label{3.15}
\widehat{a}(t)=\widetilde{\varphi}(t)+\frac{1}{R}\int_{0}^{T}\widetilde{M}(t,s)[(\gamma -1)\widehat{a}(s)+\eta ]ds,
\end{equation}%
where
\begin{align}\label{3.16}
\widetilde{\varphi}(t)&=\widehat{\varphi}(t)+\int_{0}^{t}\widetilde{P}(t,s)\widehat{\varphi}
(s)ds, \nonumber \\
\widetilde{M}(t,s)&=M(t,s)+\int_{0}^{t}\widetilde{P}(t,r)M(r,s)dr,
\end{align}
and
\begin{eqnarray}\label{3.17}
\widetilde{P}(t,s) = \sum\limits_{k=1}^{\infty
}\widetilde{\Lambda}^{k}(t,s), \quad \widetilde{\Lambda}^{1}(t,s)
=\widehat{f}(t,s), \quad
\widetilde{\Lambda}^{k+1}(t,s)=\int_{s}^{t}\widehat{f}(t,r)\widetilde{\Lambda}^{k}(r,s)dr.
\end{eqnarray}%
Note that the technique applied here is similar to that in (\ref{3.5.1}). To summarize, we have the following theorem.
\begin{theorem}
Suppose $\varphi,b,c,f$ are bounded and continuous in $t.$ If
\[
\sup_{t\in [0,T]}\int_{0}^{T}\left\vert (\gamma-1)\widetilde{M}(t,s)\right\vert ^{2}ds<R,
\]
then (\ref{3.14}) admits a unique solution $a \in C[0,T]$.
\end{theorem}
Before concluding this section, let us make some remarks on
(\ref{3.6}), (\ref{3.9}) and (\ref{3.14}). Firstly, note that $\overline{u}_i(t)$ depends on the path of $\overline{x}_i$ in
$[t,T]$. Moreover, $\widehat{c}$ is a general non-separable function
of $(t,s)$. In particular, such a general feature can naturally
degenerate and transform into the form with Riccati equation
involved, see Section 4.1 below. Secondly, equation (\ref{3.9}) is a
new type of stochastic equation with double integrals, which is
essentially different from the corresponding equation in stochastic differential equations (SDEs) case.
Since it has both the characters of Volterra and Fredholm equation,
we then call it a stochastic Volterra-Fredholm equation. We will
study its solvability under certain conditions in next section by noting our main concern here is NCE equation. Thirdly, we express the NCE
equation in our setting by means of deterministic Volterra-Fredholm
equation. This idea is also different to other literature on SDEs.

\subsection{Asymptotic equilibrium analysis}

Given the NCE equation, we need to discuss the asymptotic
equilibrium property of the associated decentralized control
strategies. To start, let us first introduce the definition of
$\epsilon$-Nash equilibrium:
\begin{definition}
Given a set of controls $\widehat{u}_i\in
L^2_{\mathcal{F}}(0,T;\dbR)$ with $i=1,2,\cdots,N$, if for any $i$,
$1\leq i\leq N,$ $\mathcal{J}_i(\widehat{u}_i,\widehat{u}_{-i})\leq
\mathcal{J}_i(u_i,\widehat{u}_{-i})+\epsilon, $ where $\epsilon>0$,
$\widehat{u}_{-i}=(\widehat{u}_1,\cdots,\widehat{u}_{i-1},\widehat{u}_{i+1},\cdots,\widehat{u}_{N})$,
$u_i\in L^2_{\mathcal{F}}(0,T;\dbR)$, then we call this set of
controls $\widehat{u}_i$, $1\leq i\leq N$ an $\epsilon$-Nash
equilibrium with respect to costs $\mathcal{J}_i$ with $1\leq i\leq
N$.
\end{definition}
Suppose $\widehat{a}$ is the solution of NCE equation (\ref{3.14}), then we can
define the decentralized control $\widehat{u}_i$ for agent $\mathcal{A}_i$ as
follows:
\begin{equation}\label{3.18}
\widehat{u}_{i}(t)=-\frac{1}{R}\dbE^{\mathcal{F}_{t}}\!\!\int_{t}^{T}\widehat{c}%
(s,t)[x_{i}(s)-\gamma \widehat{a}(s)-\eta]ds
\end{equation}%
where $x_i$ is the solution of (\ref{2.1}) corresponding to
$\widehat{u}_i$. Note that $\widehat{u}_i$ is different from
the above $\overline{u}_i$ in (\ref{3.6}). Substituting (\ref{3.18})
into equation (\ref{2.1}), together with the transformation in
(\ref{3.8}), we get
\begin{align}\label{3.19}
x_{i}(t)=& \; \widehat{\varphi }(t)+\widehat{\sigma }_i%
(t)+\int_0^t\widehat{f}(t,s)x^{N}(s)ds\nonumber\\
&- \int_{0}^{t}\widehat{c}(t,s)\frac{1}{R}\dbE^{\mathcal{F}_{s}}\!\!\int_{s}^{T}\widehat{c}%
(r,s)[x_{i}(r)-\gamma \widehat{a}(r)-\eta]drds,
\end{align}%
where $\widehat{\varphi}$, $\widehat{\sigma}_i$ and $\widehat{f}$
are defined in (\ref{3.8.1}). On the other hand, by replacing $a$ by
$\widehat{a}$ in (\ref{3.9}), we get
\begin{align}\label{3.20}
\overline{x}_{i}(t) = & \; \widehat{\varphi }(t)+\widehat{\sigma_i}
(t)-\frac{1}{R}\int_{0}^{T}\int_{0}^{s\wedge t}\widehat{c}(t,r)\widehat{c}
(s,r)\dbE^{\mathcal{F}_{r}}\overline{x}_{i}(s)drds  \nonumber \\
& +\int_0^t\widehat{f}(t,s)\widehat{a}(s)ds+\frac{1}{R}\int_{0}^{T}\int_{0}^{s\wedge t}
\widehat{c}(t,r)\widehat{c}(s,r)dr\cdot \lbrack \gamma \widehat{a}(s)+\eta ]ds.
\end{align}
To derive the asymptotic equilibrium, we need first to prove
several lemmas. In the following, we denote by $C$ a generic
positive constant independent of $N$ that may change from line to
line.
\begin{lemma}\label{lemma6-1}Suppose
\begin{eqnarray}\label{6.1}
x(t)=\varphi (t)+\int_{0}^{t}A(t,s)x(s)ds-\frac{1}{R}\int_{0}^{T}\int_{0}^{s%
\wedge t}B(t,r)B(s,r)\dbE^{\mathcal{F}_{r}}x(s)drds,
\end{eqnarray}
where $\varphi (\cdot )\in L_{\mathcal{F}}^{2}(0,T;\dbR),$ $A,B$ are
deterministic and bounded functions such that
\begin{eqnarray}\label{6.2}
T\int_{0}^{T}\int_{0}^{T}\int_{0}^{s\wedge t}|\widehat{B}%
(t,r)B(s,r)|^{2}drdsdt<\frac{R^{2}}{3}.
\end{eqnarray}
Here we denote by
\begin{eqnarray}\label{6.3}
&&\widehat{B}(t,r)=B(t,r)+\int_{r}^{t}P_0(t,v)B(v,r)dv,\quad
t,v\in[0,T],\nonumber\\
&&P_0(t,v) =\sum\limits_{k=1}^{\infty}\Lambda^{k}(t,v), \quad \Lambda^{1}(t,t) =A(t,v), \quad
\Lambda^{k+1}(t,v)=\int_{v}^{t}A(t,r)\Lambda^{k}(r,v)dr .
\end{eqnarray}
Then (\ref{6.1}) admits a unique solution $x(\cdot)\in L_{%
\mathcal{F}}^{2}(0,T;\dbR)$ satisfying
\begin{eqnarray}\label{6.4}
\dbE\int_{0}^{T}|x(s)|^{2}ds\leq 6\dbE\int_{0}^{T}|\widehat{\varphi}(s)|^{2}ds,
\end{eqnarray}
where
\begin{eqnarray*}
\widehat{\varphi }(t) &=&\varphi (t)+\int_{0}^{t}P_0(t,s)\varphi
(s)ds,\quad t\in[0,T].
\end{eqnarray*}
\end{lemma}
\begin{proof}
Given $y(\cdot )\in L_{\mathcal{F}}^{2}(0,T;\dbR),$ consider the
following Volterra equation%
\begin{eqnarray}\label{6.5}
x(t)=\varphi (t)+\int_{0}^{t}A(t,s)x(s)ds-\frac{1}{R}\int_{0}^{T}\int_{0}^{s%
\wedge t}B(t,r)B(s,r)\dbE^{\mathcal{F}_{r}}y(s)drds,
\end{eqnarray}
which obviously admits a unique solution $x(\cdot)$ under above
assumptions. By using Lemma 1.1 in
\cite{Y 1996}, we can rewrite (\ref{6.5}) as%
\[
x(t)=\widehat{\varphi }(t)-\frac{1}{R}\int_{0}^{T}\int_{0}^{s\wedge t}
\widehat{B}(t,r)B(s,r)\dbE^{\mathcal{F}_{r}}y(s)drds.
\]
Then using the idea of contraction mapping and Lemma \ref{lemma-1},
together with above requirements, we can obtain the existence of
$x(\cdot)$ in $L_{\mathcal{F}}^{2}(0,T;\dbR)$ satisfying
(\ref{6.4}).
\hfill
\end{proof}

\begin{remark}
Some comment to (\ref{6.1}). To be shown later, it plays some key role in our remaining analysis. Since our first
concern here is the solvability of (\ref{6.1}) under
certain conditions. In such sense, we do not pursue the most general assumptions here and we hope to weaken these conditions in our future work.
\end{remark}

By comparing (\ref{3.14}) and (\ref{3.20}), we have

\begin{lemma}\label{lemma-2}
Suppose NCE equation (\ref{3.14}) admits a unique solution
$\widehat{a}(\cdot)$, and (\ref{6.2}), (\ref{6.3}) hold with $A=0,$
$B=\widehat{c}$,
 then
\begin{eqnarray}\label{3.22}
\dbE\int_0^T|\widehat{a}(s)-\overline{x}^{N}(s)|^2ds=O\Big(\frac{1}{N}\Big),
\end{eqnarray}
where
$\overline{x}^{N}(t)=\frac{1}{N}\sum\limits_{i=1}^{N}\overline{x}_{i}(t)$
and $\overline{x}_i$ satisfies (\ref{3.20}).
\end{lemma}
\begin{proof}
From (\ref{3.14}) and (\ref{3.20}), we know the difference between
$\overline{x}_i$ and $\widehat{a}$ satisfies the type of Equation
(\ref{6.1}) with $A=0,$ $B=\widehat{c}$,
$\varphi=\widehat{\sigma}_i$. Then by means of (\ref{6.4}),
\begin{eqnarray*}
\dbE\int_0^T|\overline{x}^{N}(t)-\widehat{a}(t)|^{2}dt &\leq &C\dbE\int_0^T\left\vert \frac{1}{N}%
\sum_{i=1}^{N}\widehat{\sigma }_{i}(t)\right\vert ^{2}dt.
\end{eqnarray*}%
It then follows from the boundedness of $\sigma$ and $P$ in
(\ref{3.8.1}), together with the independence of $W_i$
($i=1,2,\cdots,N$) that
\[
\dbE\int_{0}^{T}|\overline{x}^{N}(t)-\widehat{a}(t)|^{2}dt=O\left(\frac{1}{N}\right),
\]
which implies the desired result. \hfill
\end{proof}

\begin{lemma}\label{lemma-3}
Suppose NCE equation (\ref{3.14}) admits a solution $\widehat{a}$,
(\ref{6.2}) and (\ref{6.3}) hold for two cases of $A=\widehat{f}$,
$B=\widehat{c},$ and $A=0,$ $B=\widehat{c}$, then
\begin{eqnarray}\label{3.23}
\dbE\int_0^T|\widehat{a}(s)-x^{N}(s)|^2ds+\sup\limits_{1\leq i\leq
N}\dbE\int_0^T|x_i(t)-\overline{x}_i(t)|^2dt=O\Big(\frac 1 N\Big)
\end{eqnarray}
where $x^{N}(t)=\frac{1}{N}\sum\limits_{i=1}^{N}x_{i}(t),$ $x_i$
satisfies (\ref{3.19}).
\end{lemma}
\begin{proof}
Comparing (\ref{3.19}) and (\ref{3.20}), we can rewrite $x_i(t)$ as
\begin{align*}
x_i(t)=& \; \overline{x}_i(t)+\int_0^t\widehat{f}(t,s)\Big[x^{N}(s)-\widehat{a}(s)\Big]ds \\
&-\frac{1}{R}\int_0^T\int_0^{s\wedge t}\widehat{c}(t,r)\widehat{c}(s,r)\dbE^{\mathcal{F}_r}[x_i(s)-\overline{x}_i(s)]drds.
\end{align*}
Therefore,
\begin{align*}
x^{N}(t)-\widehat{a}(t)=& \; \overline{x}^{N}(t)-\widehat{a}(t)+\int_0^t\widehat{f}(t,s)\Big[x^{N}(s)-\widehat{a}(s)\Big]ds\nonumber\\
& - \frac 1 R \int_0^T\int_0^{s\wedge
t}\widehat{c}(t,r)\widehat{c}(s,r)\dbE^{\mathcal{F}_r}[x^N(s)-\widehat{a}(s)]drds\nonumber\\
&+\frac 1 R \int_0^T\int_0^{s\wedge t}\widehat{c}(t,r)\widehat{c}(s,r)\dbE^{\mathcal{F}_r}[\overline{x}^N(s)-\widehat{a}(s)]drds.
\end{align*}
By using Lemma \ref{lemma6-1}, Lemma \ref{lemma-2}, together with
the boundedness of $\widehat{f}$ and $\widehat{c}$, we can obtain
the desired result (\ref{3.23}).
\hfill
\end{proof}

We also need the following result.
\begin{theorem}\label{theorem3.2}
Suppose $\varphi$, $f$, $b$ and $c$ are bounded and continuous in
$t$, NCE equation (\ref{3.14}) admits a unique continuous solution
$\widehat{a}\in C[0,T]$, $\overline{u}_i$ is the optimal control
with state equation (\ref{3.1}), cost functional (\ref{3.2}) that is
parameterized by $\widehat{a}$, (\ref{6.2}) and (\ref{6.3}) hold for two
cases of $A=\widehat{f}$, $B=\widehat{c},$ and $A=0,$
$B=\widehat{c}$. Then for $1\leq i\leq N$, we have
\begin{eqnarray}
\mid\mathcal{J}_i(\widehat{u}_i,\widehat{u}_{-i})-\overline{J}_i(\overline{u}_i)\mid=O\left(\frac
{1} {\sqrt{N}}\right).
\end{eqnarray}
\end{theorem}
\begin{proof}
First from the definition of $\mathcal{J}_i$ in (\ref{2.2}) and
$\overline{J}_i$ in (\ref{3.2}),
\begin{align*}
\mathcal{J}_{i}(\widehat{u}_{i},\widehat{u}_{-i})-\overline{J}_{i}(%
\overline{u}_{i})
\leq & \; C\dbE\int_{0}^{T}\left\vert R|\widehat{u}_{i}(t)|^{2}-R|\overline{u}%
_{i}(t)|^{2}\right\vert dt \\
&+C\dbE\int_{0}^{T}\left\vert |x_{i}(t)-\gamma x^{N}(t)-\eta |^{2}-|\overline{%
x}_{i}(t)-\gamma \widehat{a}(t)-\eta |^{2}\right\vert dt \\
\leq & \; C\dbE\int_{0}^{T}R\left[ |\widehat{u}_{i}(t)-\overline{u}%
_{i}(t)|^{2}+2|\overline{u}_{i}(t)|\cdot|\widehat{u}_{i}(t)-\overline{u}%
_{i}(t)|\right] dt \\
&+C\dbE\int_{0}^{T}|x_{i}(t)-\overline{x}_{i}(t)-\gamma (x^{N}(t)-\widehat{a}%
(t))|^{2}dt \\
&+C\dbE\int_{0}^{T}|\overline{x}_{i}(t)-\gamma \widehat{a}(t)-\eta
|\cdot |x_{i}(t)-\overline{x}_{i}(t)-\gamma
(x^{N}(t)-\widehat{a}(t))|dt,
\end{align*}
where $x_{i}$ and $\overline{x}_{i}$ satisfy (\ref{3.19}) and
(\ref{3.20}) respectively. It follows from (\ref{3.3}) and
(\ref{3.18}) that
\begin{align*}
\dbE\int_{0}^{T}|\widehat{u}_{i}(t)-\overline{u}_{i}(t)|^{2}dt
\leq & \; \frac{1}{R^{2}}\dbE\int_{0}^{T}\dbE^{\mathcal{F}_{t}}\left\vert
\int_{t}^{T}\widehat{c}(s,t)[x_{i}(s)-\overline{x}_{i}(s)]ds\right\vert ^{2}dt \\
\leq & \; C\dbE\int_{0}^{T}|x_{i}(s)-\overline{x}_{i}(s)|^{2}ds.
\end{align*}
Note that equation (\ref{3.20}) is one special case of (\ref{6.1})
with $A=0,$ $B=\widehat{c}$. Under above requirement, then from
inequality (\ref{6.4}) above we have the boundedness of
$\overline{x}_i$ in $L^2_{\mathcal{F}}(0,T;\dbR)$. Recalling
(\ref{3.6}) with $a=\widehat{a}$, we can also get similar
boundedness result for $\overline{u}_i$. As a result, by Schwarz's
inequality,
\begin{align*}
\dbE\int_{0}^{T}|\overline{u}_{i}(t)||\widehat{u}_{i}(t)-\overline{u}_{i}(t)|dt
\leq & \left[ \dbE\int_{0}^{T}|\overline{u}_{i}(t)|^{2}dt\right] ^{\frac{1}{2}}
\left[\dbE\int_{0}^{T}|\widehat{u}_{i}(t)-\overline{u}_{i}(t)|^{2}dt\right] ^{\frac{1}{2}} \\
\leq & \; C\left[ \dbE\int_{0}^{T}|x_{i}(s)-\overline{x}_{i}(s)|^{2}ds\right] ^{\frac{1}{2}}.
\end{align*}
On the other hand,
\begin{eqnarray*}
\dbE\int_{0}^{T}|x_{i}(t)-\overline{x}_{i}(t)-\gamma
(x^{N}(t)-\widehat{a}(t))|^{2}dt \leq
C\dbE\int_{0}^{T}|x_{i}(t)-\overline{x}_{i}(t)|^{2}dt+C\dbE\int_{0}^{T}|x^{N}(t)-\widehat{a}(t)|^{2}dt.
\end{eqnarray*}
Similarly, we can use Schwarz's inequality to estimate
\begin{eqnarray*}
&&\dbE\int_{0}^{T}|\overline{x}_{i}(t)-\gamma \widehat{a}(t)-\eta
|\cdot |x_{i}(t)-\overline{x}_{i}(t)-\gamma
(x^{(N)}(t)-\widehat{a}(t))|dt\\
&&\leq C\left[\dbE\int_{0}^{T}|x_{i}(t)-\overline{x}_{i}(t)|^{2}dt\right]^{\frac 1 2}+C\left[\dbE%
\int_{0}^{T}|x^{N}(t)-\widehat{a}(t)|^{2}dt\right]^{\frac 1 2}.
\end{eqnarray*}
Therefore, by using Lemma \ref{lemma-3}, we obtain the desired result.
\hfill
\end{proof}

 In the following part, given two processes $\xi _{a}$, $\xi _{b}$ in $L^2_{\mathcal{F}}(0,T;\dbR)$, we denote by
\begin{eqnarray}\label{6.6}
\epsilon _{a}=\left( \dbE\int_{0}^{T}|\xi _{a}(t)|^{2}dt\right)^{\frac{1}{2}}<\infty , \quad
\epsilon _{b}=\left( \dbE\int_{0}^{T}|\xi _{b}(t)|^{2}dt\right)^{\frac{1}{2}}<\infty .
\end{eqnarray}
First let us recall the optimal control problem of which optimal control $%
\overline{u}_{i}$ and optimal state equation $\overline{x}_{i}$ are
given by (\ref{3.6}) and (\ref{3.9}) with $a=\widehat{a}.$ For
comparison, we also discuss one perturbed version with state
equation
\begin{align}
x_{i}(t) = & \;\varphi (t)+\int_{0}^{t}b(t,s)x_{i}(s)ds+\int_{0}^{t}f(t,s)%
\widehat{a}(s)ds \nonumber \\
&+\int_{0}^{t}g(t,s)\xi
_{a}(s)ds+\int_{0}^{t}c(t,s)u_{i}(s)ds+\int_{0}^{t}\sigma
(t,s)dW_{i}(s), \label{6.7}
\end{align}
and the cost functional
\[
J_{i}^{\xi }(u_{i})=\dbE\int_{0}^{T}\left[ |x_{i}(t)-\gamma
\widehat{a}(t)-\eta +\xi _{b}(t)|^{2}+R|u_{i}(t)|^{2}\right]dt.
\]%
Here $g$ is bounded deterministic function. Then we have the
following result.
\begin{lemma}\label{lemma6.2}
Suppose $\overline{u}_{i}$ is an optimal control given in
(\ref{3.6}) and
\begin{eqnarray}\label{6.8}
\dbE\int_{0}^{T}|x_{i}(t)|^{2}dt+\dbE\int_{0}^{T}|u_{i}(t)|^{2}dt\leq
C_{0}.
\end{eqnarray}
where $ x_{i},$ $u_{i}$ are defined in (\ref{6.7}), $C_{0}$ is a
fixed constant. Moreover, $\widehat{a}$ is continuous and bounded,
(\ref{6.2}) and (\ref{6.3}) hold for $A=0,$ $B=\widehat{c},$ then we
have
\[
J_{i}^{\xi }(u_{i})\geq
\overline{J}_{i}(\overline{u}_{i})-C(1+\epsilon_b)(\epsilon_a+\epsilon_b).
\]
\end{lemma}
\begin{proof}
 For convenience, we can rewrite (\ref{6.7}) as
\begin{align*}
x_{i}(t) = & \; \varphi (t)+\int_{0}^{t}b(t,s)x_{i}(s)ds+\int_{0}^{t}f(t,s)\widehat{a}(s)ds+\int_{0}^{t}c(t,s)u_{i}^{\prime }(s)ds \\
&+\int_{0}^{t}g(t,s)\xi _{a}(s)ds+\int_{0}^{t}c(t,s)\overline{u}_{i}(s)ds+\int_{0}^{t}\sigma (t,s)dW_{i}(s).
\end{align*}
Here $u_{i}^{\prime }(s)=u_{i}(s)-\overline{u}_{i}(s),$
\[
\overline{u}_{i}(s)=-\frac{1}{R}\dbE^{\mathcal{F}_{t}}\int_{t}^{T}\widehat{c}%
(r,t)[x_{i}(r)-\gamma \widehat{a}(r)-\eta ]dr,
\]%
where $x_{i}$ satisfies (\ref{6.7}), $\widehat{c}$ is given in
(\ref{3.7}). It then follows from inequality (\ref{6.8}) that
$\dbE\int_{0}^{T}|u_{i}^{\prime }(t)|^{2}dt\leq C$. Next we
construct one auxiliary optimal control problem with state equation
and cost functional as
\begin{align}
\widetilde{x}_{i}(t) = & \; \varphi (t)+\int_{0}^{t}b(t,s)\widetilde{x}_{i}(s)ds+\int_{0}^{t}f(t,s)\widehat{a}(s)ds \nonumber \\
&+\int_{0}^{t}c(t,s)u_{i}''(s)ds+\int_{0}^{t}c(t,s)\overline{u}_{i}(s,\widetilde{x}_{i})ds+\int_{0}^{t}\sigma (t,s)dW_{i}(s), \label{6.9}
\end{align}
and
\[
\widetilde{J}_{i}(u_{i}'')=\dbE\int_{0}^{T}\left[|\widetilde{x}_{i}(t)-\gamma \widehat{a}(t)-\eta |^{2}+R|u_{i}''(t)+\overline{u}_{i}(t,\widetilde{x}_{i})|^{2}\right]dt.
\]
where $u''_{i}$ is the control variable. For $t\in[0,T]$,
$\overline{u}_{i}(t,\widetilde{x}_{i})$ in (\ref{6.9}) is defined as
\[
\overline{u}_{i}(t,\widetilde{x}_{i})=-\frac{1}{R}\dbE^{\mathcal{F}_{t}}\!\!\int_{t}^{T}\widehat{c}(r,t)[\widetilde{x}_{i}(r)-\gamma \widehat{a}(r)-\eta ]dr.
\]
Obviously, $\widetilde{J}_{i}(u_{i}'')$ attains its
minimum when $u_{i}''=0,$ and the corresponding cost functional equals to $%
\overline{J}_{i}(\overline{u}_{i}).$ On the other hand, by Lemma \ref{lemma6-1},
\[
M_{i}'(t)=\int_{0}^{t}c(t,s)u_{i}'(s)ds+\int_{0}^{t}b(t,s)M_{i}'(s)ds-\frac{1}{R}\int_{0}^{t}c(t,s)\dbE^{\mathcal{F}_{t}}\int_{t}^{T}\widehat{c}(r,t)M_{i}'(r)drds
\]
admits a unique solution $M_{i}'$ such that
\[
\dbE\int_{0}^{T}|M_{i}^{\prime }(t)|^{2}dt\leq 6\dbE\int_{0}^{T}\left\vert\int_{0}^{t}c(t,s)u_{i}'(s)ds\right\vert^{2}dt\leq C.
\]
If we take $u_{i}''=u_{i}'(s)=u_{i}(s)-\overline{u}_{i}(s)$, and
$\widetilde{x}_{i}'$ is the corresponding solution of (\ref{6.9}), then $\widetilde{x}_{i}'=\overline{x}_{i}(t)+M_{i}'(t),$ where $\overline{x}_{i}(t)$ is the
optimal state in (\ref{3.9}) with $a=\widehat{a}$. Moreover,
recalling the boundedness of $\overline{x}_i$ in
$L^2_{\mathcal{F}}(0,T;\dbR)$, we also have
$\dbE\int_{0}^{T}|\widetilde{x}_{i}'(t)|^{2}dt\leq C$.
Since $x_{i}$ and $\widetilde{x}_{i}'$ are given in
(\ref{6.7}) and (\ref{6.9}), thus their difference
$d_{i}=x_{i}-\widetilde{x}_{i}'$ satisfy
\[
d_{i}(t)=\int_{0}^{t}g(t,s)\xi _{a}(s)ds+\int_{0}^{t}b(t,s)d_{i}(s)ds-\frac{1}{R}\int_{0}^{t}c(t,s)\dbE^{\mathcal{F}_{t}}\!\!\int_{t}^{T}\widehat{c}(r,t)d_{i}(r)drds.
\]
Under above requirements, from Lemma \ref{lemma6-1}, we have
\[
\dbE\int_{0}^{T}|x_{i}(t)-\widetilde{x}_{i}^{\prime }(t)|^{2}dt\leq 6\dbE\int_{0}^{T}\left\vert\int_{0}^{t}g(t,s)\xi _{a}(s)ds\right\vert^{2}dt\leq C\epsilon _{a}^{2}.
\]
Following the similar arguments in \cite{N-H 2012}, we should have
\[
|J_{i}^{\xi}(u_{i})-\widetilde{J}_{i}(u_{i}')| \leq
C(1+\epsilon_b)(\epsilon _{a}+\epsilon _{b}),
\]
where $C$ does not depend on $\xi_a$ and $\xi_b$. After combining the
estimate of $\widetilde{J}_{i}(u_{i}')\geq
\overline{J}_i(\overline{u}_i)$, we can get the desired result.
\hfill
\end{proof}

The next lemma is concerned about the boundedness of
$\mathcal{J}_{i}(\widehat{u}_i,\widehat{u}_{-i})$.
\begin{lemma}\label{lemma6.3}
Suppose $\widehat{a}$ is solution of NCE equation (\ref{3.14}),
(\ref{6.2}) and (\ref{6.3}) hold for two cases of $A=\widehat{f}$,
$B=\widehat{c}$, and $A=0$, $B=\widehat{c}$. Then
$\mathcal{J}_{i}(\widehat{u}_i,\widehat{u}_{-i})\leq C$, where $C$
does not depend on $N$.
\end{lemma}
\begin{proof}
If we can prove
\[
\dbE\int_{0}^{T}|\widehat{u}_{i}(s)|^{2}ds\leq C,\quad
\dbE\int_{0}^{T}|x_{i}(s)|^{2}ds\leq C,
\]
where $x_i$ satisfies (\ref{3.19}), $C$ does not depend on $N,$ then
the result holds directly. From (\ref{3.19}), we have
\begin{align}
x^{N}(t) = & \; \widehat{\varphi }(t)+\int_{0}^{t}\widehat{f}(t,s)x^{N}(s)ds+\frac{1}{N}\sum_{l=1}^{N}\widehat{\sigma }_{l}(t) \nonumber \\
&-\int_{0}^{t}\widehat{c}(t,s)\frac{1}{R}\dbE^{\mathcal{F}s}\int_{s}^{T}\widehat{c}(r,s)[x^{N}(r)-\gamma \widehat{a}(r)-\eta ]drds, \label{6.10}
\end{align}
where $x^{N}(t)=\frac{1}{N}\sum\limits_{l=1}^{N}x_{l}(t).$ Note that
(\ref{6.10}) is one special case of equation (\ref{6.1}),
(\ref{6.2}) and (\ref{6.3}) hold for $A=\widehat{f}$,
$B=\widehat{c}$, it then follows from Lemma \ref{lemma6-1} that
\begin{align}
\dbE\int_{0}^{T}|x^{N}(s)|^{2}ds \leq & \; C\dbE\int_{0}^{T}|\widehat{\varphi}(s)|^{2}ds+C\dbE\int_{0}^{T}\left\vert \frac{1}{N}\sum_{l=1}^{N}\widehat{\sigma}_{l}(t)\right\vert ^{2}dt \nonumber \\
& +C\dbE\int_{0}^{T}|\widehat{a}(s)|^{2}ds+C. \label{6.11}
\end{align}
As to $x_i$ in equation (\ref{3.19}), using (\ref{6.11}) and
conditions (\ref{6.2}), (\ref{6.3}) with $A=0$, $B=\widehat{c}$, we
should have $\dbE\int_{0}^{T}|x_{i}(s)|^{2}ds\leq C$ where $C$ is
independent of $N$. At last, we know that
\begin{align*}
\dbE\int_{0}^{T}|\widehat{u}_{i}(t)|^{2}dt
= & \; \frac{1}{R^{2}}\dbE\int_{0}^{T}\left\vert\int_{t}^{T}\widehat{c}(r,t)[x_{i}(r)-\gamma\widehat{a}(r)-\eta ]dr\right\vert ^{2}dt \\
\leq & \; \frac{1}{R^{2}}\int_{0}^{T}\int_{t}^{T}|\widehat{c}(r,t)|^{2}drdt\cdot \dbE\int_{0}^{T}[x_{i}(r)-\gamma\widehat{a}(r)-\eta ]^{2}dr.
\end{align*}
Hence $\dbE\int_{0}^{T}|\widehat{u}_{i}(t)|^{2}dt$ is bounded, too.
Then the desired result holds.
\hfill
\end{proof}

If we apply control $\widehat{u}_j$ in (\ref{3.18}) to all the player except $\mathcal{A}_i,$ that is, for $j\neq i$,
\begin{align}
x_{j}(t) = & \; \varphi(t)+\int_{0}^{t}b(t,s)x_{j}(s)ds+\int_{0}^{t}f(t,s)[x^{N}(s)-\widehat{a}(s)]ds \nonumber \\
& +\int_{0}^{t}f(t,s)\widehat{a}(s)ds+\int_{0}^{t}c(t,s)\widehat{u}_{j}(s)ds+\int_{0}^{t}\sigma (t,s)dW_{j}(s), \label{6.12}
\end{align}
while for player $i$,
\begin{align}
x_{i}(t) = & \; \varphi(t)+\int_{0}^{t}b(t,s)x_{i}(s)ds+\int_{0}^{t}c(t,s)u_{i}(s)ds \nonumber \\
&+\int_{0}^{t}\sigma (t,s)dW_{i}(s)+\int_{0}^{t}f(t,s)x^{N}(s)ds, \label{6.13}
\end{align}
then similar as Lemma \ref{lemma-3} above, we also have the
following result.
\begin{lemma}\label{lemma6.4}
Suppose $\widehat{a}$ is the solution of NCE equation (\ref{3.14}),
(\ref{6.2}) and (\ref{6.3}) hold for two cases of $A=\widehat{f}$,
$B=\widehat{c}$ and $A=0,$ $B=\widehat{c}$. Moreover,
$(1+K_2)K_1\frac{360}{N^{2}R^2}L^2<\frac 1 2$, where
\[
L=\int_0^T\int_0^t|\widehat{c}(t,s)|^2dsdt, \quad
K_1=\int_0^T\int_0^t|\widehat{f}(t,s)|^2dsdt,\quad
K_2=\int_0^T\int_0^t|P_0(t,s)|^2dsdt,
\]
$P_0$ is defined in (\ref{6.3}) with $A=\widehat{f}$. Then we have
\begin{eqnarray*}
\dbE\int_0^T|\widehat{a}(s)-x^{N}(s)|^2ds=O\Big(\frac{1}{N}\Big),
\end{eqnarray*}
where $x^{N}(t)=\frac{1}{N}\sum_{l=1}^{N}x_{l}(t),$ $x_l$
satisfies (\ref{6.12}) and (\ref{6.13}).
\end{lemma}
\begin{proof}
For convenience, we also rewrite (\ref{6.13}) as
\begin{align}
x_{i}(t) = & \; \varphi (t)+\int_{0}^{t}b(t,s)x_{i}(s)ds+\int_{0}^{t}f(t,s)\widehat{a}(s)ds+\int_{0}^{t}c(t,s)\widehat{u}_{i}(s)ds \nonumber \\
& +\int_{0}^{t}\sigma (t,s)dW_{i}(s)+\int_{0}^{t}f(t,s)[x^{N}(s)-\widehat{a}(s)]ds \nonumber \\
& +\int_{0}^{t}c(t,s)[u_{i}(s)-\widehat{u}_{i}(s)]ds, \label{6.14}
\end{align}
where $\widehat{u}_i$ is defined as
\begin{eqnarray}\label{6.15}
\widehat{u}_{i}(t)=-\frac{1}{R}\dbE^{\mathcal{F}_{t}}\!\!\int_{t}^{T}\widehat{c}(r,t)[x_{i}(r)-\gamma \widehat{a}(r)-\eta ]dr,
\end{eqnarray}
and $x_{i}$ satisfies (\ref{6.13}). After some direct calculations,
\begin{align*}
x^{N}(t) = & \; \widehat{\varphi }(t)+\int_{0}^{t}\widehat{f}(t,s)\widehat{a}(s)ds+\frac{1}{N}\sum_{l=1}^{N}\widehat{\sigma }_{l}(t) \\
& -\int_{0}^{t}\widehat{c}(t,s)\frac{1}{R}\dbE^{\mathcal{F}s}\int_{s}^{T}\widehat{c}(r,s)[x^{N}(r)-\gamma \widehat{a}(r)-\eta ]drds \\
& +\int_{0}^{t}\widehat{f}(t,s)[x^{N}(s)-\widehat{a}(s)]ds+\frac{1}{N}\int_{0}^{t}\widehat{c}(t,s)[u_{i}(s)-\widehat{u}_{i}(s)]ds,
\end{align*}
where $\widehat{\varphi}$, $\widehat{f}$, $\widehat{\sigma}_l$ are
defined in (\ref{3.8.1}). Using Fubini theorem, it also follows from
(\ref{3.9}) (with $a=\widehat{a}$) that
\begin{align*}
\overline{x}^{N}(t) = & \; \widehat{\varphi }(t)+\int_{0}^{t}\widehat{f}(t,s)\widehat{a}(s)ds+\frac{1}{N}\sum_{l=1}^{N}\widehat{\sigma }_{l}(t) \\
&
-\int_{0}^{t}\widehat{c}(t,s)\frac{1}{R}\dbE^{\mathcal{F}s}\!\!\int_{s}^{T}\widehat{c}(r,s)[\overline{x}^{N}(r)-\gamma
\widehat{a}(r)-\eta ]drds,
\end{align*}
therefore we have
\begin{align}
x^{N}(t)-\widehat{a}(t) = & \; \overline{x}^{N}(t)-\widehat{a}(t)+\int_{0}^{t}\widehat{f}(t,s)[x^{N}(s)-\widehat{a}(s)]ds \nonumber \\
& -\int_{0}^{t}\widehat{c}(t,s)\frac{1}{R}\dbE^{\mathcal{F}s}\!\!\int_{s}^{T}\widehat{c}(r,s)[x^{N}(r)-\widehat{a}(r)]drds \nonumber \\
& +\int_{0}^{t}\widehat{c}(t,s)\frac{1}{R}\dbE^{\mathcal{F}s}\!\!\int_{s}^{T}\widehat{c}(r,s)[\overline{x}^{N}(r)-\widehat{a}(r)]drds \nonumber \\
& +\frac{1}{N}\int_{0}^{t}\widehat{c}(t,s)[u_{i}(s)-\widehat{u}_{i}(s)]ds. \label{6.16}
\end{align}
Note that if $A=0$, $B=\widehat{c}$, then (\ref{6.2}) and
(\ref{6.3}) imply $\overline{x}^N$ is bounded in
$L^2_{\mathcal{F}}(0,T;\dbR).$ Similarly, if (\ref{6.2}) and
(\ref{6.3}) hold with  $A=\widehat{f}$, $B=\widehat{c}$, then using
Lemma \ref{lemma6-1} to (\ref{6.16}), we have
\begin{eqnarray}\label{6.17}
&&\dbE\int_{0}^{T}|x^{N}(t)-\widehat{a}(t)|^{2}dt \nonumber\\ &&\leq
(1+K_2)\left[(36+12R^2)\dbE\int_{0}^{T}|\overline{x}^{N}(t)-\widehat{a}(t)|^{2}dt+\frac{36}{N^{2}}L\cdot
\dbE\int_{0}^{T}[u_{i}(s)-\widehat{u}_{i}(s)]^{2}ds\right],
\end{eqnarray}
where $L$ and $K_2$ are defined before.
On the other hand, from Lemma \ref{lemma6.3},
$\mathcal{J}_i(\widehat{u}_i,\widehat{u}_{-i})$ is bounded, it then
suffices to consider $u_i$ in (\ref{6.13}) such that
$\mathcal{J}_i(u_i,\widehat{u}_{-i})\leq
\mathcal{J}_i(\widehat{u}_i,\widehat{u}_{-i})$, which implies that
$\dbE\int_0^T|u_i(t)|^2dt\leq C$.  As to $\widehat{u}_i$ in
(\ref{6.15}), we have
\begin{align*}
\dbE\int_{0}^{T}|\widehat{u}_{i}(t)|^{2}dt = & \; \frac{1}{R^{2}}\dbE\int_{0}^{T}\left\vert\int_{t}^{T}\widehat{c}(r,t)[x_{i}(r)-\gamma \widehat{a}(r)-\eta ]dr\right\vert ^{2}dt \\
\leq & \; \frac{1}{R^{2}}L\cdot \dbE\int_{0}^{T}[x_{i}(r)-\gamma \widehat{a}(r)-\eta ]^{2}dr,
\end{align*}
therefore,
\begin{align}
\dbE\int_{0}^{T}[u_{i}(s)-\widehat{u}_{i}(s)]^{2}ds
\leq & \; 2\dbE\int_{0}^{T}|u_{i}(s)|^{2}ds+2\dbE\int_{0}^{T}|\widehat{u}_{i}(s)|^{2}ds \nonumber \\
\leq & \; C+\frac{2}{R^{2}}L\cdot \dbE\int_{0}^{T}|x_{i}(r)|^{2}dr, \label{6.18}
\end{align}
where $C$ does not depend on $N$. Substituting (\ref{6.18}) into
(\ref{6.17}), we have
\begin{eqnarray}\label{6.19}
&&\dbE\int_{0}^{T}|x^{N}(t)-\widehat{a}(t)|^{2}dt\nonumber\\
&& \leq
(1+K_2)(36+12R^2)\dbE\int_{0}^{T}|\overline{x}^{N}(t)-\widehat{a}(t)|^{2}dt\nonumber
\\
&&+\frac{C}{N^{2}}+ (1+K_2)\frac{72}{N^{2}R^2}L^2\cdot
\dbE\int_{0}^{T}[u_{i}(s)-\widehat{u}_{i}(s)]^{2}ds.
\end{eqnarray}
After some transformations, we can rewrite (\ref{6.13}) as
\begin{align*}
x_{i}(t) = & \; \widehat{\varphi }(t)+\int_{0}^{t}\widehat{f}(t,s)[x^{N}(s)-\widehat{a}(s)]ds \\
& +\int_{0}^{t}\widehat{f}(t,s)\widehat{a}(s)ds+\int_{0}^{t}\widehat{c}(t,s)u_{i}(s)ds+\widehat{\sigma }_{i}(t),
\end{align*}
therefore,
\begin{align}
\dbE\int_{0}^{T}|x_{i}(t)|^{2}dt
\leq & \; 5\dbE\int_{0}^{T}|\widehat{\varphi }(t)|^{2}dt+5\int_{0}^{T}\int_{0}^{t}|\widehat{f}(t,s)|^{2}dsdt\cdot\dbE\int_{0}^{T}|x^{N}(s)-\widehat{a}(s)|^{2}ds \nonumber \\
& +5\int_{0}^{T}\int_{0}^{t}|\widehat{f}(t,s)|^2dsdt\cdot \int_{0}^{T}|\widehat{a}(s)|^{2}ds \nonumber \\
& +5\int_{0}^{T}\int_{0}^{t}|\widehat{c}(t,s)|^{2}dsdt\cdot\dbE\int_{0}^{T}|u_{i}(s)|^{2}ds+\dbE\int_{0}^{T}|\widehat{\sigma}_{i}(t)|^{2}dt. \label{6.20}
\end{align}
Since $\widehat{\varphi}$, $\widehat{f}$, $\widehat{a}$, $\sigma$
are bounded, substituting (\ref{6.19}) into (\ref{6.20}) yields
\begin{eqnarray*}
\dbE\int_{0}^{T}|x_{i}(t)|^{2}dt\leq C+C\dbE\int_{0}^{T}|\overline{%
x}^{N}(t)-\widehat{a}(t)|^{2}dt+
(1+K_2)K_1\frac{360}{N^{2}R^2}L^2\cdot
\dbE\int_{0}^{T}[x_{i}(s)]^{2}ds.
\end{eqnarray*}
So if $(1+K_2)K_1\frac{360}{N^{2}R^2}L^2<\frac 1 2,$ then we have
\[
\dbE\int_{0}^{T}|x_{i}(t)|^{2}dt\leq C+C\dbE\int_{0}^{T}|\overline{x}^{N}(t)-%
\widehat{a}(t)|^{2}dt,
\]
which implies the boundedness of $\dbE\int_{0}^{T}|x_{i}(t)|^{2}dt$.
By combining (\ref{6.19}), (\ref{6.20}) and Lemma \ref{lemma-2}, we obtained the desired result.
\hfill
\end{proof}
\begin{remark}
Compared with Lemma 3.4 above, our result here seems to be more
general because $u_i$ in (\ref{6.13}) is not necessary to be
$\widehat{u}_i$ in (\ref{3.18}). Moreover, our result also
generalizes the corresponding result in \cite{N-H 2012} in
that our state equation can degenerate to the linear SDEs there.
In addition, our approach here is also different from the
one used in \cite{N-H 2012}. At last, although the conditions imposed on above coefficients $K_1$, $K_2$ and $L$ seem to be technical, it
has two interesting features. First, the above inequality holds true in general because $N$ is always sufficiently large. Second,
such condition is not required in the nontrivial case $\widehat{f}=0.$
\end{remark}
Now let us state the main result in this subsection.
\begin{theorem}\label{theorem3.3}
Suppose $\varphi$, $f$, $b$ and $c$ are bounded and continuous in
$t$, the NCE equation (\ref{3.14}) admits a unique continuous
solution $\widehat{a}\in C[0,T]$. (\ref{6.2}) and (\ref{6.3}) hold
with $A=\widehat{f}$, $B=\widehat{c}$ and $A=0,$ $B=\widehat{c}$.
Moreover, $(1+K_2)K_1\frac{360}{N^{2}R^2}L^2<\frac 1 2$, where $L,\
K_1,\ K_2$ are defined in Lemma \ref{lemma6.4}. Then the set of
control $\widehat{u}_i$ in (\ref{3.18}) with $1\leq i\leq N$ for $N$
players is an $\epsilon$-Nash equilibrium, i.e., for $1\leq i\leq
N$,
$$\mathcal{J}_i(\widehat{u}_i,\widehat{u}_{-i})-\epsilon\leq
\inf\limits_{u_i\in
L^2_{\mathcal{F}}[0,T]}\mathcal{J}_i(u_i,\widehat{u}_{-i})\leq
\mathcal{J}_i(\widehat{u}_i,\widehat{u}_{-i}).
$$
\end{theorem}
\begin{proof}
If we denote by $\xi_a(t)=x^N(t)-\widehat{a}(t),$
$\xi_b(t)=\gamma(\widehat{a}(t)-x^N(t))$, then the dynamics and cost functional
of $\mathcal{A}_i$ can be rewritten as
\begin{align}
x_{i}(t) = & \; \varphi(t)+\int_{0}^{t}b(t,s)x_{i}(s)ds+\int_{0}^{t}f(t,s)\widehat{a}(s)ds \nonumber \\
& +\int_{0}^{t}f(t,s)\xi_{a}(s)ds+\int_{0}^{t}c(t,s)u_{i}(s)ds+\int_{0}^{t}\sigma(t,s)dW_{i}(s), \label{6.21}
\end{align}
and
\[
J_{i}(u_{i})=\dbE\int_{0}^{T}\left[ |x_{i}(t)-\gamma\widehat{a}(t)-\eta +\xi _{b}(t)|^{2}+R|u_{i}(t)|^{2}\right]dt.
\]
From Lemma \ref{lemma6.4}, we have
$\epsilon_a+\epsilon_b=O(\frac{1}{\sqrt{N}})$, and
$\dbE\int_0^T|x_i(t)|^2dt+\dbE\int_0^T|u_i(t)|^2dt\leq C,$ where $C$
is independent of $N$, $\epsilon_a$, $\epsilon_b$ are defined in
(\ref{6.6}), it then follows from Lemma \ref{lemma6.4}, Theorem
\ref{theorem3.2} that
\[
\mathcal{J}_i(u_i,\widehat{u}_{-i})\geq \overline{J}_i(\overline{u}_i)-O\left(\frac{1}{\sqrt{N}}\right) \geq \mathcal{J}_i(\widehat{u}_i,\widehat{u}_{-i})-O\left(\frac{1}{\sqrt{N}}\right).
\]
The conclusion is thus proved.
\hfill
\end{proof}

\section{Some special cases}

In this section, we want to show that our general results can
include some special cases that are important for real
applications. For sake of simplicity, we suppose $f=0$ in the
following.

\subsection{The case of stochastic differential equation}
In this case, $\varphi ,b,c,\sigma $ are independent of $t$, and
Equation (\ref{2.1}) becomes
\begin{eqnarray}\label{4.1}
x_{i}(t)=x(0)+\int_{0}^{t}b(s)x_{i}(s)ds+%
\int_{0}^{t}c(s)u_{i}(s)ds+\int_{0}^{t}\sigma(s)dW_i(s),
\end{eqnarray}
while the cost functional (\ref{2.2}) still keeps the same. Most related literatures on mean field LQG games are discussed
by the above linear SDEs. For $s,t\in[0,T],$ we have
\[
\Lambda^{1}(s,t)=b(t), \quad
\Lambda^{k+1}(s,t)=b(t)\left(\int_{t}^{s}b(r)dr\right)^{k}, \quad
P(s,t)=b(t)e^{\int_{t}^{s}b(u)du}.
\]
Therefore,
\begin{eqnarray*}
& & \widehat{\varphi }(t) = x(0)e^{\int_{0}^{t}b(r)dr}, \quad \widehat{\sigma}_i(t) =\int_{0}^{t}\sigma (s)e^{\int_{s}^{t}b(r)dr}dW_{i}(s), \\
& & \widehat{c}(s,t)=c(t)e^{\int_{t}^{s}b(r)dr}, \quad M(t,s)=\int_{0}^{s\wedge t}|c(r)|^{2}e^{\int_{r}^{t}b(v)dv+\int_{r}^{s}b(v)dv}dr.
\end{eqnarray*}%
The NCE equation in this setting takes the following rather simple form
\begin{eqnarray}\label{4.2}
\widehat{a}(t)=\widehat{\varphi
}(t)+\frac{1}{R}\int_{0}^{T}\int_{0}^{s\wedge
t}|c(r)|^{2}e^{\int_{r}^{t}b(v)dv+\int_{r}^{s}b(v)dv}dr\cdot\big\{\gamma
\widehat{a}(s)+\eta-\widehat{a}(s)\big\}ds.
\end{eqnarray}

\begin{theorem}\label{theorem4-1}
Suppose that $b$, $c$ are bounded functions such that
\[
\sup_{t\in [0,T]}\int_{0}^{T}\Big[(\gamma -1)\int_{0}^{s\wedge t}|c(r)|^{2}e^{\int_{r}^{t}b(v)dv+\int_{r}^{s}b(v)dv}dr\Big]^{2}ds<R.
\]
Then NCE equation (\ref{4.2}) admits a continuous solution.
\end{theorem}
We define by
\begin{eqnarray}\label{4.3}
\widehat{u}_i(t)=-\frac{1}{R}c(t)\dbE^{\mathcal{F}_t}\!\!\int_t^Te^{\int_t^sb(r)dr}[x_i(s) -\gamma\widehat{a}(s)-\eta]ds,\quad 1\leq i\leq N,
\end{eqnarray}
where $\widehat{a}$ is the solution of NCE equation and
\begin{align}
x_{i}(t) = & \; x(0)e^{\int_{0}^{t}b(r)dr}+\int_{0}^{t}\sigma (s)e^{\int_{s}^{t}b(r)dr}dW_{i}(s) \nonumber \\
& -\frac{1}{R}\int_{0}^{t}|c(s)|^2\dbE^{\mathcal{F}_{s}}\int_{s}^{T}e^{\int_{s}^{r}b(v)dv+\int_s^tb(r)dr} [x_{i}(r)-\gamma\widehat{a}(r)-\eta]drds, \label{4.4}
\end{align}
then we get the following asymptotic equilibrium analysis in present setting.

\begin{theorem}\label{theorem4-2}
Suppose $b$, $c$ and $\sigma$ are bounded and deterministic, the NCE
equation admits a solution $\widehat{a}\in C[0,T]$. Moreover,
\begin{eqnarray*}
T\int_{0}^{T}\int_{0}^{T}\int_{0}^{s\wedge t}|c(r)|^4e^{2\int_r^tb(v)dv+2\int_r^sb(v)dv}drdsdt<\frac{R^{2}}{3}.
\end{eqnarray*}
Then the set of controls defined in (\ref{4.3}) for $N$ players is
an $\epsilon$-Nash equilibrium.
\end{theorem}

\begin{remark}
One can express $\widehat{u}_i$ in (\ref{4.3}) as
$\widehat{u}_i(t)=\frac 1 R c(t)[-P'(t)x_i(t)+\delta(t)]$ where $P'$
and $\delta$ satisfy Riccati equation and backward ordinary
differential equation respectively, see for example, \cite{H 2010},
\cite{H-C-M 2007} and \cite{N-H 2012}. Such procedure is different
from ours here. More precise, it makes use the advantage of state equation and optimal control (note that $s$ and $t$ are separable). However, such skills are hard to be
applied in more general model such as the state equations with delay
below. In such sense, our approach applied here is more flexiable
which can fill this technical gap.
\end{remark}

\subsection{SDEs with delay in the state process}

In this case, suppose cost functional is (\ref{2.2}) and the state equation is described by%
\begin{equation}\label{4.5}
dx_{i}(t)=\left[A(t)x_{i}(t-h)+\int_{t-h}^{t}B(t,s)x_{i}(s)ds\right]dt+C(t)u_{i}(t)dt+D(t)dW_i(t),
\end{equation}
where $h>0$, $x_{i}(t)=k(t),$ $t\in \lbrack -h,0],$ $k(t)$ is
bounded. Hence the delay term
appears in the state process. By introducing function $\Phi_1 $ as%
\begin{eqnarray}\label{4.6}
\left\{\begin{array}{rl} \dsp\frac{\partial \Phi_1(t,s)}{\partial t}
& = \dsp A(t)\Phi_1(t-h,s)+\int_{t-h}^{t}B(t,r)\Phi_1(r,s)dr,\ t\geq
0, \\ [5mm] \Phi_1(0,0) & = 1, \quad \Phi_1(t,s)=0, \quad t<0,
\end{array}\right.
\end{eqnarray}%
we transform (\ref{4.5}) into (see \cite{L 1973}, \cite{O-Z
2010})
\begin{eqnarray}\label{4.7}
x_{i}(t)=\psi(t)+\int_{0}^{t}K(t,s)u_i(s)ds+\int_{0}^{t}\Phi_1
(t,s)D(s)dW_i(s),
\end{eqnarray}%
where $K(t,s)=\Phi_1(t,s)C(s),$ and
\begin{eqnarray*}
\psi(t)=\Phi_1(t,0)k(0)+\int_{-h}^{0}\left[ \Phi_1
(t,s+h)A(s+h)+\int_{0}^{h}\Phi_1(t,u)B(u,s)du\right] k(s)ds .
\end{eqnarray*}%
It follows that the above $\Phi_1$ is bounded and continuous in
$t$. On the other hand, (\ref{4.7}) is one special case of Equation (\ref{2.1}) by letting $b=0,$ hence for $t,s\in[0,T]$,
\begin{eqnarray*}
& & P(t,s) = 0, \quad \widehat{c}(s,t)=\Phi_1(t,s)C(s), \\
& & \widehat{\varphi }(t) = \psi(t), \quad \widehat{\sigma}_{i}(t)=\int_{0}^{t}\Phi_1(t,s)D(s)dW_{i}(s),\\
& & M(t,s) = \int_{0}^{s\wedge t}\Phi_1(t,r)\Phi_1(s,r)|C(r)|^{2}dr.
\end{eqnarray*}
The NCE equation in this situation becomes
\begin{eqnarray}\label{4.8}
\widehat{a}(t)=\psi(t)+\frac{1}{R}\int_{0}^{T}\int_{0}^{s\wedge
t}\Phi_1(t,r)\Phi(s,r)|C(r)|^{2}dr\cdot\big\{\gamma
\widehat{a}(s)+\eta-\widehat{a}(s)\big\}ds.
\end{eqnarray}

\begin{theorem}\label{theorem4-3}
Suppose $A,$ $B,$ $C$ are bounded deterministic functions such that
\[
\sup_{t\in [0,T]}\int_{0}^{T}\Big[(\gamma -1)\int_{0}^{s\wedge
t}\Phi_1(t,r)\Phi_1(s,r)|C(r)|^{2}dr\Big] ^{2}ds<R,
\]%
where $\Phi_1$ is defined by (\ref{4.6}), then NCE equation
(\ref{4.8}) admits a continuous solution.
\end{theorem}
Given $\widehat{a}$ being a solution of NCE equation, if we define
\begin{eqnarray}\label{4.9}
\widehat{u}_i(t)=-\frac 1 R
C(t)\dbE^{\mathcal{F}_t}\int_t^T\Phi_1(s,t) [x_i(s) -\gamma
\widehat{a}(s)-\eta]ds,\quad 1\leq i\leq N,
\end{eqnarray}
where
\begin{align*}
x_{i}(t) = & \; \psi(t)+\int_{0}^{t}\Phi_1(t,s)D(s)dW_{i}(s) \\
& -R^{-1}\int_{0}^{T}\int_{0}^{s\wedge t}\Phi_1(t,r)\Phi_1(s,r)|C(r)|^{2}\dbE^{\mathcal{F}_{r}}[x_{i}(s)-\gamma \widehat{a}(s)-\eta]drds,
\end{align*}
then we get the following asymptotic equilibrium analysis in such setting.

\begin{theorem}\label{theorem4-4}
Suppose that $A,$ $B,$ $C$ and $D$ are bounded deterministic functions,
the NCE equation admits a consistent solution $\widehat{a} \in C[0,T]$. Moreover,
\begin{eqnarray*}
T\int_{0}^{T}\int_{0}^{T}\int_{0}^{s\wedge t}|C(r)|^4\Phi_1(t,r)\Phi_1(s,r)drdsdt<\frac{R^{2}}{3}.
\end{eqnarray*}
Then the set of controls defined in (\ref{4.9}) for $N$ players is
an $\epsilon$-Nash equilibrium.
\end{theorem}

\subsection{SDEs with delay in the control}

In this case, suppose the cost functional is (\ref{2.2}) and the
state equation is described by
\begin{equation}\label{4.10}
dx_{i}(t)=A(t)x_{i}(t)dt+C(t)u_i(t-h)dt+D(t)dW_i(t),
\end{equation}
where $h>0,$ $x_{i}(t)=k(t),$ $t\in [-h,0]$, $C(t)=0$ with $t<h.$
The delayed term appears in the control variable. Similarly, by
introducing function $\Phi_2 $ as
\begin{eqnarray}\label{4.11}
\frac{\partial \Phi_2(t,s)}{\partial t} =A(t)\Phi_2(t,s) \quad
\Phi_2(s,s) =1,\quad t,s\in[0,T],
\end{eqnarray}
we can transform equation (\ref{4.10}) into
\begin{eqnarray}\label{4.12}
x_{i}(t)=\psi(t)+\int_{0}^{t}K(t,s)u_i(s)ds+\int_{0}^{t}\Phi_2(t,s)D(s)dW_i(s),
\end{eqnarray}
where
\[
\psi(t)=\Phi_2(t,0)k(0), \quad K(t,s)=\Phi_2(t,s+h)C(s+h), \quad t,s\in[0,T].
\]
Therefore, (\ref{4.12}) is one special case of (\ref{2.1}) with $b=0,$ and
\begin{eqnarray*}
& & P(t,s) = 0,\quad \widehat{c}(s,t)=\Phi_2(s,t+h)C(t+h), \\
& & \widehat{\varphi }(t) = \Phi_2(t,0)k(0),\quad \widehat{\sigma }_i(t)=\int_{0}^{t}\Phi_2(t,s)D(s)dW_{i}(s),\\
& & M(t,s) = \int_{0}^{s\wedge t}\Phi_2(t,r+h)\Phi_2(s,r+h)|C(r+h)|^{2}dr.
\end{eqnarray*}
The NCE equation in this situation becomes
\begin{equation}\label{4.13}
\widehat{a}(t)=\psi (t)+\frac{1}{R}\int_{0}^{T}\int_{0}^{s\wedge
t}\Phi_2(t,r+h)\Phi_2(s,r+h)|C(r+h)|^{2}dr \cdot
\big\{\gamma\widehat{a}(s)+\eta-\widehat{a}(s)\big\}ds.
\end{equation}

\begin{theorem}\label{theorem4-5}
Suppose $A$, $C$ and $D$ are bounded functions such that
\[
\sup_{t\in [0,T]}\int_{0}^{T}\left[ (\gamma -1)\int_{0}^{s\wedge
t}\Phi_2(t,r+h)\Phi_2(s,r+h)|C(r+h)|^{2}dr \right] ^{2}ds<R,
\]%
then NCE equation (\ref{4.13}) admits a unique solution.
\end{theorem}

Given $\widehat{a}$ being the NCE solution, we define %
\begin{eqnarray}\label{4.14}
\widehat{u}_{i}(t)=-\frac{1}{R}C(t+h)\dbE^{\mathcal{F}_{t}}\3n\int_{t}^{T}\Phi_2
(s,t+h)[x_{i}(s)-\gamma \widehat{a}(s)-\eta]ds,\quad t\in[0,T],
\end{eqnarray}
while $x_{i}$ satisfies
\begin{align*}
x_{i}(t) = & \; \Phi_2(t,0)k(0)+\int_{0}^{t}\Phi_2(t,s)D(s)dW_{i}(s) \\
& -R^{-1}\int_{0}^{T}\int_{0}^{s\wedge t}\Phi_2(t,r+h)\Phi_2(s,r+h)|C(r+h)|^{2}\dbE^{\mathcal{F}_{r}}[x_{i}(s)-\gamma\widehat{a}(s)-\eta]drds,
\end{align*}
then we get the following asymptotic equilibrium analysis,
\begin{theorem}\label{theorem4-6}
Suppose that $b$ and $c$ are bounded, the NCE equation (\ref{3.14})
admits a unique continuous solution $\widehat{a}\in C[0,T]$.
Moreover,
\begin{eqnarray*}
T\int_{0}^{T}\int_{0}^{T}\int_{0}^{s\wedge
t}|C(r+h)|^4\Phi_2(t,r+h)\Phi_2(s,r+h)drdsdt<\frac{R^{2}}{3},
\end{eqnarray*}
where $\Phi_2$ is defined in (\ref{4.11}). Then the set of controls
defined in (\ref{4.14}) for $N$ players is an $\epsilon$-Nash
equilibrium.
\end{theorem}

\section{Concluding remark} Herein, we investigate a class of mean-field LQG games
where state equation is some stochastic Volterra integral system.
The NCE consistency condition is derived based on some Fredholm
equations and the $\epsilon$-Nash equilibrium property of
decentralized controls is also established. Our work is the first
attempt to the LQG games with stochastic integral system and there
arise various research directions upon it. On one hand, it is
possible to include the Volterra integral kernel into the cost
functional to be minimized. This provides some potential way to
study the LQG game with time inconsistency. On the other hand, this
paper also considers the mean-field games with stochastic delayed
system. Actually, more extensive research can be given in this
research line and it is anticipated some new consistency conditions
can be given which depend on the delay characters.

\end{document}